\documentclass[11pt]{amsart}
\usepackage{amsfonts}
\usepackage{amsmath}
\usepackage{amssymb, esint}
\usepackage{amsthm,color}
\usepackage{enumerate}
\usepackage{mathtools}

\usepackage[normalem]{ulem}

\usepackage{cancel}
\usepackage{longtable}

\usepackage{geometry}
\usepackage[hidelinks, bookmarksdepth=3]{hyperref}

\newcommand{\R}{{\mathbb R}}

\newcommand{\N}{{\mathbb N}}

\newcommand{\cE}{{\mathcal E}}

\newcommand{\cH}{{\mathcal H}}

\newcommand{\Ss}{{\mathbb S}}

\def\0{{\mathbf 0}}

\def\0{{\mathbf 0}}
\def\loc{{\textup{loc}}}
\def\id{{\rm id}}

\def \mb{\mathbb}

\def\loc{{\textup{loc}}}

\newcommand{\e}{\varepsilon}
\newcommand{\vp}{\varphi}
\newcommand{\osc}{\operatornamewithlimits{osc}}

\newcommand{\supp}{\operatorname{supp}}

\newcommand{\diam}{\operatorname{diam}}
\newcommand{\dist}{\operatorname{dist}}

\newcommand{\Hess}{\operatorname{Hess}}

\newcommand{\esssup}{\operatornamewithlimits{ess\,sup}}

\newcommand\norm[1]{\left\Arrowvert {#1} \right\Arrowvert}

\theoremstyle{plain}
\newtheorem{thm}{Theorem}[section]

\newtheorem{cor}[thm]{Corollary}
\newtheorem{lem}[thm]{Lemma}

\theoremstyle{definition}

\theoremstyle{remark}

\newtheorem*{claim*}{Claim}

\newtheorem{rem}[thm]{Remark}

\numberwithin{equation}{section}

\title[Constraint Maps ... Alt--Phillips case]{Partial regularity for minimizing constraint maps \\ for the Alt-Phillips energy}

\author[R.\ Ziganshina]{Rada Ziganshina}
\email{radaz@kth.se@kth.se}
\address{Department of Mathematics, KTH Royal Institute of Technology, Sweden}

\begin{document}

\begin{abstract}
In this paper, we establish an  \texorpdfstring{$\varepsilon$-}regularity theorem for minimizers of an Alt-Phillips type functional subject to constraint maps. We prove  that under sufficiently small energy, the minimizers exhibit regularity, and hence  proving the smoothness of these maps. From here, we  bootstrap to optimal regularity.
\end{abstract}

\maketitle
\setcounter{tocdepth}{1}



\section{Introduction}\label{sec:intro}

\subsection{Background}

The problem of partial regularity for  constraint maps was first investigated in \cite{D}, and has been recently revisited and further developed  (\cite{FKS}, \cite{FGKS1}, \cite{FGKS2}). This renewed interest has led to studies of regularity of the projection and distance maps, the structure of the free boundary, and possible singularities occurring on or near the free boundary. These aspects have been primarily studied in the context of the classical Dirichlet energy, with some extensions to the Alt–Caffarelli functional \cite{FGKS1}. In this work, we initiate an investigation for functionals with source terms, related to 
the scalar Alt–Phillips energy.
It should be noted that the results in \cite{GG, GG2}, where the Alt–Phillips energy and a broader class of functionals were analyzed in vectorial settings, were obtained in the absence of constraints.

To formulate our problem,  let $\Omega$ be a bounded domain in $\R^n$, $n \geq 1$, and consider a smooth\footnote{The smoothness of $\partial M$ can be made exact, depending on $\gamma$, in the functional, and the analysis in this paper. However, we assume it is $C^\infty$ for simplicity}
domain $M$ in $\R^m$, $m \geq 2$ with compact, connected complement. Throughout this paper, we shall write by $\rho$ the distance function to $M^c$. Consider the Alt-Phillips type energy,
\begin{equation*}
\cE_{\lambda,\gamma}[u]= \int_{\Omega} \frac{1}{2}|Du|^2 + {\lambda}(\rho\circ u)^\gamma\,dx,\quad \lambda >  0, \,\gamma \in (0,2). 
\end{equation*}
We shall call $u\in W^{1,2}(\Omega;\overline M)$ a minimizing constraint map for $\cE_{\lambda,\gamma}$, if 
$$
\cE_{\lambda,\gamma}[u]\leq \cE_{\lambda,\gamma}[v],\quad\forall v\in W_u^{1,2}(\Omega;\overline M). 
$$
The case with $\lambda>0$, $\gamma = 0$ corresponds to Alt-Caffarelli energy, which is treated in \cite{FGKS1}, while the case with $\lambda = 0$ (hence any $\gamma$) corresponds to the Dirichlet energy, which is treated in \cite{FGKS2}. 

Our main result Theorem \ref{thm:e-reg} below, is an $\varepsilon$-regularity theorem, which  
states that if the energy $\cE_{\lambda,\gamma}$ for a minimizer  is small, then the minimizer is  smooth to an optimal degree.

In our analysis, we encountered two main difficulties compared to previous approaches. 
First, techniques based solely on constraint maps were insufficient; the additional term 
in the Alt--Phillips functional introduces a component in the Euler--Lagrange equation 
that is difficult to control for $\gamma \in (0,1)$, yet is crucial for the regularity 
estimates. Second, while the results in \cite{GG2} address the Alt--Phillips term, 
they do not account for the constraints essential to our setting.

\begin{thm}[\texorpdfstring{$\e$-regularity}{e-regularity}]
\label{thm:e-reg}
Let $u\in W^{1,2}(\Omega;\overline M)$ be a minimizing constraint map for $\cE_{\lambda,\gamma}$ with $\rho\circ u\leq d$ a.e.\ in a ball $B_{2R}(x_0)\subset\Omega$. For every $\sigma\in(0,1)$, there are constants $\e > 0$ small and $c > 1$ large, all depending at most on $n$, $m$, $\partial M$, $\lambda$,  $\gamma$, and $\sigma$, such that if 
$$
(2R)^{2-n} \int_{B_{2R}(x_0)} |Du|^2\,dx \leq \e \quad\text{and}\quad R^2\lambda d^\gamma \leq \e,
$$
then $u\in C^{1,\alpha}(B_R(x_0);\overline M)$ with $\alpha = \min\{\frac{\gamma}{2-\gamma},\sigma\}$ and 
    $$
    R\| Du \|_{L^\infty(B_R(x_0))} + R^{1+\alpha} [ Du ]_{C^{0,\alpha}(B_R(x_0))}\leq c.
    $$
Moreover, if $\gamma\in[1,2)$, and $d \leq d_0$ where $d_0$ is such that $\rho\in C^\infty(B_{d_0}(\partial M))$, one has $u\in C^{1,1}(B_R(x_0);\overline M)$ and 

$$
\| D^2 u\|_{L^\infty(B_R(x_0))} \leq c\norm{Du}_{L^\infty(B_{2R}(x_0))}^2 + \frac{c}{R^2}\lambda \gamma (d_0)^{\gamma-1}
$$

\end{thm}

With Theorem \ref{thm:e-reg} at hand, we establish an optimal partial regularity theorem for Alt-Phillips constraint maps, significantly improving Luckhaus' result \cite{L}

\begin{cor}
    \label{cor:partial}
    Let $u\in W^{1,2}\cap L^\infty(\Omega;\overline M)$ be a minimizing constraint map for $\cE_{\lambda,\gamma}$, where $\lambda > 0$, and $\gamma \in (0,2)$ are given constants. Then there is a relatively closed set $\Sigma\subset\Omega$, with $\dim_\cH(\Sigma) \leq n-3$, such that $u\in C_\loc^{1,\alpha}(\Omega\setminus\Sigma)$, where $\alpha = \min\{\frac{\gamma}{2-\gamma},\sigma\}$ for any $\sigma\in(0,1)$. Moreover, if $\gamma\geq 1$, one has $u\in C_\loc^{1,1}(\Omega\cap B_\delta(\partial u^{-1}(M))\setminus\Sigma)$ for some $\delta > 0$. 
\end{cor}


\subsection{Notation}\label{sec:notation1}

For the reader's convenience, we gather here a list of notation to be used throughout this paper.

\begin{longtable}[l]{l l }
\,\,\,\,$\Omega$ &\qquad  a bounded domain in $\R^n$ ($n\geq 1$)  \medskip \\ 
\,\,\,\,$M$  &\qquad a $C^\infty $ domain in $\R^m$  ($m\geq 2$)  \medskip \\ 
\,\,\,\,$u$  &\qquad $(u^1, \cdots , u^m)$    \medskip \\
\,\,\,\,$W^{1,p}(\Omega;\overline M)$&\qquad $\{u\in W^{1,p}(\Omega;\R^m): u\in \overline M\text{ a.e.\ in }\Omega\}$      \medskip \\
\,\,\,\,$W^{1,p}_g(\Omega;\overline M)$&\qquad  the Dirichlet class $W^{1,p}(\Omega;\overline M)\cap (g+W^{1,p}_0(\Omega;\R^m))$, for $g\in W^{1,p}(\Omega;\overline M)$   \medskip \\
\,\,\,\,$\mb S^{m-1}$  &\qquad the unit sphere in $\R^m$ \medskip \\ 
\,\,\,\,$A_y (\cdot , \cdot)$  &\qquad     the second fundamental form of $\partial M$ at $y$   \medskip \\   
\,\,\,\,$B_d(E)$  &\qquad   the $d$-neighborhood of a set $E$, $\bigcup_{y\in E}\{x : |x-y| < d\}$
    \medskip\\  
\,\,\,\,$\rho$  &\qquad   signed distance function to $\partial M$, taking positive values in $M$   \medskip \\ 
 \,\,\,\,$d_0$  &\qquad   the width of the tubular neighborhood of $\partial M$ for which $\rho\in C^\infty(B_{d_0}(\partial M))$\medskip \\ 
\,\,\,\,$\nu$  &\qquad inward unit normal to $\partial M$ \medskip \\   
\,\,\,\,$\Pi$   &\qquad   $(\Pi^1, \cdots , \Pi^m)$, nearest point projection from $B_{d_0}(\partial M)$ onto $\partial M$
       \medskip\\ 
\,\,\,\,$\nu\otimes \nu$   &\qquad        the tensor product matrix with entries $(\nu^i\nu^j)_{ij}$ \medskip\\   
\,\,\,\,$\xi^\top$  &\qquad   $
  (I - \nu\otimes \nu)\xi
$,   \ orthogonal projection of $\xi$ into $T(\partial M)$      \medskip \\ 
\,\,\,\,$\cE_0[u]$  &\qquad $\int_\Omega \frac{1}{2}|Du|^2 \, dx$   \medskip \\ 
\,\,\,\,$\cE_{\lambda,\gamma}[u]$  &\qquad $\int_\Omega \frac{1}{2}|Du|^2 + \lambda(\rho\circ u)^\gamma \, dx$   \medskip \\ 
\,\,\,\,$\Sigma(u)$   &\qquad  $ \{ x\in\Omega: u\text{ is not continuous at }x\} $    \medskip \\   
\,\,\,\,$\cH^k$  &\qquad  $k$-dimensional Hausdorff measure    \medskip \\   
\,\,\,\,$\hbox{dim}_\cH$  &\qquad  Hausdorff dimension     \medskip\\    
\,\,\,\,$F:G$  &\qquad  $ f_\alpha^i g_\alpha^i$, \ where  $F = (f_\alpha^i)$ and  $G= (g_\alpha^i)$ are matrices in $\R^{m\times n}$    \medskip \\ 
\,\,\,\,$D^k$ &\qquad the $k$-th order differential operator in the ambient space, $D_\alpha := \frac{\partial}{\partial x_\alpha}$ \medskip \\
\,\,\,\,$\nabla^k$ &\qquad the $k$-th order differential operator in the target space, $\partial_i := \frac{\partial}{\partial y^i}$ \medskip \\
\,\,\,\,$\Hess f_y(\xi,\zeta)$  &\qquad $\partial_{ij} f(y)\xi_\alpha^i\zeta_\alpha^j$, for a function $f$ from the target space, and $\xi,\zeta \in \R^{m\times n}$ \medskip\\ 
\end{longtable}


\section{Preliminaries and basic tools}\label{sec:prelim}

Throughout this paper, we shall assume that $M$ is a $C^\infty$-domain in $\R^m$, with compact and connected complement. Here we assume the smoothness of $\partial M$ only for the sake of simplicity, even if most of our analysis works well for $\partial M$ of class $C^3$, or even $C^2$.  

Given a boundary datum $g\in W^{1,2}(\Omega;\overline M)$, the existence of minimizing constraint maps $u\in W_g^{1,2}(\Omega;\overline M)$ for $\cE_{\lambda,\gamma}$ follows from the direct method of calculus of variations. 

Of course, it is not a trivial matter to have $W^{1,2}(\Omega;\overline M)\neq\emptyset $ for $n = 2$ and it depends on the topology of $ M$ with regard  to $\Omega$.
However, $W^{1,2}(\Omega;\overline M)\neq\emptyset$ whenever $n\geq 3$; we refer interested  readers to \cite[Remark 2.9]{FGKS2} 
for the topology of the Sobolev class $W^{1,2}(\Omega;\overline M)$.

Throughout this paper, we shall consider minimizing constraint maps in $ W^{1,2}\cap L^\infty (\Omega;\overline M)$, which is not a restrictive condition as shown in the following lemma.

\begin{lem}\label{lem:bdd}
    Let $g\in W^{1,2}\cap L^\infty (\Omega;\overline M)$ be given, and let $u\in W_g^{1,2}(\Omega;\overline M)$ be a minimizing constraint map for $\cE_{\lambda,\gamma}$. Then $u(\Omega) \subset K$ up to a null set, where $K$ is the closed convex hull of $g(\Omega)\cup M^c$. In particular, as $M^c$ is compact, $|u|\in L^\infty(\Omega)$.
\end{lem}

\begin{proof}
Let $H$ be a closed half-space in $\R^m$ such that $g(\Omega)\cup M^c\subset H$. Let $\pi_H : \R^m\to H$ be the orthogonal projection map. Then for any $v\in W_g^{1,2}(\Omega;\overline M)$, we have $\pi_H\circ v \in W_g^{1,2}(\Omega;\overline M)$, as well as that $|D(\pi_H\circ v)|\leq |Dv|$ and $\rho\circ \pi_H\circ v\leq \rho\circ v$ a.e.\ in $\Omega$; as a result,  
$$
\cE_{\lambda,\gamma}[\pi_H\circ v]\leq \cE_{\lambda,\gamma}[v],\quad\forall v\in W_g^{1,2}(\Omega;\overline M).
$$
Thus, the minimality of $u\in W_g^{1,2}(\Omega;\overline M)$ for $\cE_{\lambda,\gamma}$ shows that 
$$
0 = \cE_{\lambda,\gamma}[u] - \cE_{\lambda,\gamma}[\pi_H\circ u] = \int_\Omega \frac{1}{2}(|Du|^2 - |D(\pi_H\circ u)|^2) + {\lambda}((\rho\circ u)^\gamma -(\rho\circ \pi_H\circ u)^\gamma) \,dx.
$$
Since both parentheses in the integrand are nonnegative a.e.\ in $\Omega$, we obtain $|D(\pi_H\circ u)| = |Du|$ a.e.\ in $\Omega$. In particular, denoting by $\nu_H$ the outward unit normal to $\partial H$, it follows from the last identity that 
$$
|D\dist(u, H)|^2 = |\nu_H Du|^2 = |Du|^2 - |D(\pi_H\circ u)|^2 = 0 \quad\text{a.e.\ in }\Omega. 
$$
However, as $u\in W_g^{1,2}(\Omega;\overline M)$ and $H$ is a closed half-space containing $g(\Omega)$, we have $\dist(g,H) = 0$ a.e.\ in $\Omega$, whence 
$$
\dist(u,H) = 0\quad\text{a.e.\ on }\partial\Omega,
$$
in the trace sense. Combining this with the displayed equation above, we deduce that $\dist(u,H) = 0$ a.e.\ in $\Omega$, i.e., $u(\Omega)\subset H$ up to a null set. As $H$ was an arbitrary closed half-space containing $g(\Omega)$, the assertion of this lemma follows. 
\end{proof}

Since energy minimizing maps are also stationary with respect to inner variations, we obtain a monotonicity formula for Alt-Phillips constraint maps. Note that we obtain an almost monotonicity formula, rather than an exact one, due to the presence of the extra term $\lambda(\rho\circ u)^\gamma$ in the energy functional $\cE_{\lambda,\gamma}$. This is essentially the same as Alt-Caffarelli constraint maps \cite{FGKS1}:

\begin{lem}[Almost monotonicity formula]
    \label{lem:monot}
Let $u\in W^{1,2}\cap L^\infty(\Omega;\overline M)$ be a minimizing constraint map for $\cE_{\lambda,\gamma}$. Then 
$$
s^{2-n} \int_{B_s(x_0)} |Du|^2\,dx - r^{2-n}\int_{B_r(x_0)} |Du|^2\,dx \geq 2\int_{B_s(x_0)\setminus B_r(x_0)} R^{2-n}\left| \frac{\partial u}{\partial N}\right|^2 dx - \omega_n {\lambda} d^\gamma s^2,
$$
whenever $0 < r \leq s < \dist(x_0,\partial\Omega)$, where $R := |x-x_0|$, $\partial/\partial N$ denotes the directional derivative in the radial direction $R^{-1}(x-x_0)$, and $d = \esssup_\Omega(\rho\circ u)$. Moreover, if $\lambda = 0$, and $r^{2-n}\int_{B_r(x_0)} |Du|^2\,dx = s^{2-n}\int_{B_s(x_0)} |Du|^2\,dx$, then $u$ is $0$-homogeneous about $x_0$ in $B_s(x_0)\setminus B_r(x_0)$, i.e., $(x-x_0) \cdot Du = 0$ a.e.\ in $B_s(x_0)\setminus B_r(x_0)$. 
\end{lem}

\begin{proof}
    One may follow almost verbatim the proof of \cite[Lemma 2.2]{FGKS1}. The only difference here is that at the final stage, we estimate $\lambda(\rho\circ u)^\gamma \leq \lambda d^\gamma$, which holds by assumption a.e.\ in $\Omega$. We omit the details. 
\end{proof}

Note that for $\gamma > 0$, the Alt-Phillips energy $\cE_{\lambda,\gamma}$ verifies the main assumption \cite[(A1)]{L}. Especially, with the above monotonicity formula at hand, the entire partial regularity theory in \cite{L} carries over the current setting. 

\begin{thm}[Partial regularity due to \cite{L}]
    \label{thm:e-reg-Ca}
    Let $u\in W^{1,2}(\Omega;\overline M)$ be a minimizing constraint map for $\cE_{\lambda,\gamma}$. Then there exists a closed set $\Sigma\subset\Omega$ such that $u\in C_\loc^{0,\alpha}(\Omega\setminus\Sigma;\overline M)$, for every $\alpha \in (0,1)$, and $\dim_\cH\Sigma \leq n-3$. In particular, when $n = 3$, the set $\Sigma$ is discrete. Moreover, $\lim_{r\downarrow 0} r^{2-n}\int_{B_r(x_0)} |Du|^2\,dx = 0$ for all $x_0\in \Omega\setminus\Sigma$.
\end{thm}

\begin{rem}
  Note that Theorem \ref{thm:e-reg-Ca} only provides $C^{0,\alpha}$ regularity. 
However, optimal smoothness is established in Corollary \ref{cor:partial}, 
which yields $C^{1,\alpha}$ regularity ($0 < \alpha \leq 1$), where $\alpha$ depends on the value of $\gamma$.
\end{rem}

We will later need the following results for $\cE_0$-minimizing constraint maps. 

\begin{thm}
    \label{thm:e-reg-clean}[Results for the case \texorpdfstring{$\cE_0$}{cE_o}.]
    Let $v\in W^{1,2}(B_{2R}(x_0);\overline M)$ be a minimizing constraint map for $\cE_0$. Then there exists $\e_0 > 0$, depending only on $n$ and $\partial M$, such that if $(2R)^{2-n}\int_{B_{2R}(x_0)} |Dv|^2\,dx < \e_0$, then $v\in C^{1,1}(B_R(x_0);\overline M)$ and 
    \begin{equation}
        \label{eq:e-reg-clean}
        \| D^j v \|_{L^\infty(B_R(x_0))} \leq \frac{1}{R^j},\quad j\in\{1,2\}.
    \end{equation}
    Moreover, for every $\alpha\in(0,1)$, it satisfies 
    \begin{equation}
        \label{eq:camp}
        \begin{aligned}
            \int_{B_r(y)} |Dv -(Dv)_{y,r}|^2\,dx &\leq c\bigg(\frac{r}{s}\bigg)^{n+2\alpha}\int_{B_s(y)} |Dv - (Dv)_{y,s}|^2\,dx\\
            &\quad + c\bigg(\frac{r}{R}\bigg)^{n+2\alpha}\int_{B_{2R}(x_0)} |Dv|^2\,dx,
        \end{aligned}
    \end{equation}
    whenever $0<r<s\leq R - |y-x_0|$ and $y\in B_R(x_0)$, where $c > 1$ may depend further on $\alpha$. 
\end{thm}

Estimate \eqref{eq:e-reg-clean} is provided in \cite[Theorem 2.4]{FGKS2}. While the second estimate \eqref{eq:camp} follows from \cite[Proposition 1]{L} and standard variational techniques (cf. \cite[Theorem 9.7]{GM}), we provide a detailed derivation here to ensure the exposition is self-contained and easily accessible.

\begin{proof}[Proof of \eqref{eq:camp}]
    Throughout the proof, we shall fix $\alpha \in (0,1)$, and denote by $c$ a generic positive constant that may depend at most on $n$, $\partial M$, and $\alpha$.    
    By \cite{D}, we have 
    \begin{equation}
        \label{eq:v-EL}
        \Delta v = A_v(Dv,Dv)\chi_{v^{-1}(\partial M)}\quad\text{in }B_{2R}(x_0),
    \end{equation}
    in the weak sense. Now let us fix $y\in B_R(x_0)$, and $0< r< s \leq R - |y-x_0|$. Then since $(2R)^{2-n}\int_{B_R(x_0)} |Dv|^2\,dx < \e_0$ and $B_R(y)\subset B_{2R}(x_0)$, the $\e$-regularity theory \cite{L} yields that
    \begin{equation}
        \label{eq:e-reg-L}
        \int_{B_s(y)}|Dv|^2 \,dx \leq c \bigg(\frac{s}{R}\bigg)^{n-2+2\alpha} \int_{B_{2R}(x_0)} |Dv|^2\,dx. 
    \end{equation} 
    Now let $h\in W_v^{1,2}(B_s(y);\R^m)$ be such that $|\Delta h| = 0$ in $B_s(y)$. Since each component of $h$ is harmonic in $B_s(y)$, we obtain from the usual variational comparison with harmonic functions that 
    \begin{equation}
        \label{eq:v-h-C1a}
        \int_{B_r(y)} |Dv - (Dv)_{y,r}|^2\,dx \leq c\bigg(\frac{r}{s}\bigg)^{n+2}\int_{B_s(y)} |Dv - (Dv)_{y,s}|^2\,dx + c \int_{B_s(y)} |D(v-h)|^2\,dx. 
    \end{equation}
    By \eqref{eq:v-EL}, the assumption that $\partial M$ is smooth and compact, which yields $|A_y(\xi,\xi)|\leq c|\xi|^2$ uniformly for all $y\in \partial M$ and all $\xi\in T_y(\partial M)$, as well as the energy minimality of $h$ within $W_v^{1,2}(B_s(y);\R^m)$, we obtain that 
    \begin{equation}
        \label{eq:v-h-C1a-2}
        \begin{aligned}
            \int_{B_s(y)} |D(v-h)|^2\,dx &\leq 2\int_{B_s(y)} Dv:D(v-h)\,dx \\
            & = - 2\int_{B_s(y)\cap v^{-1}(\partial M)} (v-h)\cdot A_v(Dv,Dv)\,dx \\
            &\leq c \int_{B_s(y)} |v-h||Dv|^2\,dx. 
        \end{aligned}
    \end{equation}

    We shall now use the first part of Theorem \ref{thm:e-reg-clean} that $v\in C^{1,1}(B_R(x_0);\partial M)$ with the estimate \eqref{eq:e-reg-clean}, since this part is already known, c.f. \cite[Theorem 2.4]{FGKS2}.
Let us write by $\ell$ the affine approximation of $v$ at $y$, i.e., $\ell(y) = v(y)$ and $D\ell (y) = Dv(y)$. Then by \eqref{eq:e-reg-clean}, we have 
    \begin{equation}
        \label{eq:v-l-sup}
        \sup_{B_s(y)} |v-\ell| \leq \bigg(\frac{s}{R}\bigg)^2. 
    \end{equation}
    Together with the maximum principle for harmonic functions and the fact that $h = v$ on $\partial B_s(y)$, we can deduce from \eqref{eq:v-l-sup} that
    \begin{equation}
        \label{eq:h-l-sup}
        \sup_{B_s(y)} |h-\ell| \leq \sup_{\partial B_s(y)} |v-\ell| \leq \bigg(\frac{s}{R}\bigg)^2. 
    \end{equation}
    Hence, by \eqref{eq:e-reg-L}, \eqref{eq:v-l-sup}, \eqref{eq:h-l-sup}, we deduce that 
    \begin{equation}
        \label{eq:v-h}
        \int_{B_s(y)} |v-h||Dv|^2\,dx \leq c\bigg(\frac{s}{R}\bigg)^{n+2\alpha} \int_{B_{2R}(x_0)} |Dv|^2\,dx. 
    \end{equation}
    Combining \eqref{eq:v-h-C1a-2} with \eqref{eq:v-h}, we obtain an upper bound for the rightmost term of \eqref{eq:v-h-C1a}, and proceed with the computation further as 
    $$
    \int_{B_r(y)} |Dv-(Dv)_{y,r}|^2\,dx \leq c\bigg(\frac{r}{s}\bigg)^{n+2}\int_{B_s(y)} |Dv - (Dv)_{y,s}|^2\,dx + c\bigg(\frac{s}{R}\bigg)^{n+2\alpha}\int_{B_{2R}(x_0)}|Dv|^2\,dx. 
    $$

    Our final assertion \eqref{eq:camp} now follows from the standard iteration lemma, c.f.\ \cite[Lemma 2.2]{GG}.

\end{proof}

\begin{rem}
    \label{rem:e-reg-clean}
    We shall need \eqref{eq:camp} later in Section \ref{sec:part-reg} when studying the partial regularity of minimizing constraint maps for the Alt-Phillips energy. Note that \eqref{eq:camp} is a $C^{1,\alpha}$-estimate for $\alpha < 1$. Although we have the optimal $C^{1,1}$-estimate \eqref{eq:e-reg-clean}, it is not clear to us whether we can also establish \eqref{eq:camp} for $\alpha = 1$. 
\end{rem}
As mentioned at the beginning, we assume that $\partial M$ is of class $C^\infty$. 
Throughout this paper, we shall write by $\rho$ the signed-distance map to $\partial M$, which takes positive values in $M$. As $\partial M$ is smooth, there is a tubular neighborhood of $\partial M$ of width $d_0 > 0$, denoted by $B_{d_0}(\partial M)$, in which $\rho \in C^\infty(B_{d_0}(\partial M))$. In this neighborhood, we can also consider the nearest point projection $\Pi$ onto $\partial M$. Note that $\Pi\in C^\infty({B_{d_0}(\partial M)};\partial M)$.

Let us first compute the Euler-Lagrange equation for the distance map $\rho\circ u$.

\begin{lem}\label{lem:dist-pde}
    Let $u\in W^{1,2}(\Omega;\overline M)$ be a minimizing constraint map for $\cE_{\lambda,\gamma}$. Suppose that $u\in C^1(U;{B_{d_0}(\partial M)})$ in a subdomain $U\subset\Omega$. Then $\rho\circ u \in C^1(U)$ and  
\begin{equation}\label{eq:dist-pde}
\begin{cases}
\Delta(\rho\circ u) = \Hess\rho_u(Du,Du) + \lambda\gamma(\rho\circ u)^{\gamma-1} & \text{in }U\cap u^{-1}(M), \\
\rho\circ u = |D(\rho\circ u)| = 0 & \text{on }U\setminus u^{-1}(M),
\end{cases}
\end{equation}
holds in the weak sense.  
\end{lem}

\begin{proof}
With $u\in C^1(U;{B_{d_0}(\partial M)})$ and $\rho\in C^\infty({B_{d_0}(\partial M)})\subset C^1({B_{d_0}(\partial M)})$, we have $\rho\circ u \in C^1(U)$. 
Since $\rho\circ u$ takes local minimum on $U\cap\partial u^{-1}(M)$, it follows that $|D(\rho\circ u)| = 0$ on $U\cap\partial u^{-1}(M)$. This proves the second line of \eqref{eq:dist-pde}. 

Since we assume that $u\in C^1(U;B_{d_0}(\partial M))\subset C(U;\R^m)$, the set $U\cap u^{-1}(M)$ is open. Moreover, since $u(U)\subset B_{d_0}(\partial M)$ and $\partial M$ is smooth, we have at least $\nu\circ u\in C^1(U)$. Therefore, given any (scalar)  function $\vp \in C_c^\infty(U\cap u^{-1}(M))$, we find $\delta > 0$ such that $\delta \leq \rho\circ u\leq d_0- \delta$ on $\supp\vp$.

Let $L > 0$ be such that $|\vp|\leq L$ on $\supp\vp$. Then for any $\e > 0$ small such that $2L\e \leq \delta$, we have 
$$0 <  \frac{1}{2}\delta  \leq \delta - \e L \leq  \rho \circ (u + \e \vp \nu\circ u)    
\leq d_0 - \delta + \e L < d_0, \qquad \hbox{in } \supp\vp.$$ 
Moreover, we also have $\rho\circ (u + \e \vp\nu\circ u) = \rho\circ u \geq 0$ in $\Omega\setminus\supp\vp$.

This shows that $u + \e\vp\nu\circ u \in W^{1,2}(\Omega;\overline M)$ and that $u + \e\vp\nu\circ u\in C^1(U;B_{d_0}(\partial M))$. Note further that since $\rho^\gamma \in C^2(M\cap B_{d_0}(\partial M))$ with $\nabla\rho^\gamma = \gamma \rho^{\gamma-1}\nu$ in $\{\rho > 0\}$, $\rho\circ u > 0$ in $u^{-1}(M)$, and $\supp\vp \Subset U\cap u^{-1}(M)$, we obtain  (using Taylor expansion)
\begin{equation}
    \label{eq:EL1}
    \begin{aligned}
        \int_\Omega \rho^\gamma\circ (u + \e \vp\nu\circ u)\,dx &= \int_\Omega \rho^\gamma\circ u \,dx  + \e \int_{U\cap u^{-1}(M)} \left(\gamma (\rho^{\gamma-1}\nu)\circ u\right)\cdot \left(\varphi \nu\circ u\right)\,dx + o(\e) \\
        &= \int_\Omega \rho^\gamma\circ u \,dx  + \e \int_{U\cap u^{-1}(M)} \gamma \varphi\rho^{\gamma-1}\circ u\,dx + o(\e),
    \end{aligned}
\end{equation}
as $\e \to 0$. Also, by the chain rule, we have $D_\alpha(\rho\circ u) = \nu^i\circ u D_\alpha u^i$ and $D_\alpha (\nu^i\circ u)D_\alpha u^i = \Hess\rho_u(D_\alpha u,D_\alpha u)$ a.e.\ in $U$, which shows that 
\begin{equation}
    \label{eq:EL2}
    \begin{aligned}
        \int_\Omega |D ( u + \e \vp\nu\circ u)|^2 \,dx &= \int_\Omega |Du|^2 \,dx \\
        &\quad + 2\e \int_{U\cap u^{-1}(M)} D(\rho\circ u)\cdot D\vp + \vp \Hess\rho_u(Du,Du)\,dx + o(\e). 
    \end{aligned}
\end{equation}
Recalling that $u + \e \vp\nu\circ u \in W^{1,2}(\Omega;\overline M)$ and that $\supp \vp \Subset U\cap u^{-1}(M)\subset\Omega$, the minimality of $u$ for $\cE_{\lambda,\gamma}$ now yields, along with \eqref{eq:EL1} and \eqref{eq:EL2}, that 
\begin{equation}
    \label{eq:EL3}
    \begin{aligned}
        \cE_{\lambda,\gamma}[u] &\leq \cE_{\lambda,\gamma}[u + \e\vp\nu\circ u] \\
        & \leq \cE_{\lambda,\gamma}[u] + 2 \e \int_{U\cap u^{-1}(M)} D(\rho\circ u)\cdot D\vp + \vp \Hess\rho_u(Du,Du) + \gamma \varphi \rho^{\gamma-1}\circ u\,dx + o(\e).
    \end{aligned}
\end{equation}
Hence, subtracting $\cE_{\lambda,\gamma}[u]$ from the left hand side and the rightmost side of \eqref{eq:EL3} and dividing the resulting inequality by $\e$ and letting $\e\to 0$ yields 
\begin{equation}
    \label{eq:EL4}
    \int_{U\cap u^{-1}(M)} D(\rho\circ u)\cdot D\vp + \vp \Hess\rho_u(Du,Du) + \gamma \varphi  \rho^{\gamma-1}\circ u\,dx \geq 0. 
\end{equation}
Finally, replacing $\vp$ with $-\vp$ turns the inequality in \eqref{eq:EL4} into an equality. As $\vp\in C_c^\infty(U\cap u^{-1}(M))$ was arbitrary, we arrive at the first line of \eqref{eq:dist-pde}.

\end{proof}

We can also compute the Euler-Lagrange system for the projection map $\Pi\circ u$.

\begin{lem}\label{lem:proj-pde}
    Let $u\in W^{1,2}(\Omega;\overline M)$ be a minimizing constraint map for $\cE_{\lambda,\gamma}$. If $u\in {B_{d_0}(\partial M)}$ a.e.  in a ball    $B\subset\Omega$, then $\Pi\circ u\in W^{1,2}(B;\partial M)$ and 
    \begin{equation}\label{eq:proj-pde}
    \Delta(\Pi\circ u) = -2 D(\rho\circ u)\cdot D(\nu\circ\Pi\circ u) + \Hess\Pi_u((Du)^\top,(Du)^\top)\quad\text{in } B,
    \end{equation}
    in the weak sense, where $(Du)^T$ is the tangential component of $D_\alpha u(x)$ with respect to the tangent hyperplane of $\partial M$.
\end{lem}

\begin{proof}
Fix $i\in\{1,2,\cdots,m\}$, and consider the flow $\Phi \in C^1((-\eta,\eta) \times {B_{d_0}(\partial M)}; {B_{d_0}(\partial M)})$ generated by $\nabla \Pi^i$. Since $\nu\cdot\nabla\Pi^i = 0$ in ${B_{d_0}(\partial M)}$, we have $\rho\circ \Phi(t,y) = \rho(y)$ for all $(t,y)\in(-\eta,\eta)\times {B_{d_0}(\partial M)}$. 

Let 
$\vp\in C_c^\infty(B)$ be given, and take $\e> 0$ small enough such that $\e|\vp| < \eta$ in $B$. Then we take the variation of the form $u_\e := \Phi(\e\vp, u)\in W^{1,2}(B;\overline M)$. Note $u_\e - u \in W_0^{1,2}(B;\R^m)$. By the choice of the flow $\Phi$, we observe that $\rho\circ u_\e = \rho\circ u$ a.e.\ in $B$. Thus, $\cE_{\lambda,\gamma}[u]\leq \cE_{\lambda,\gamma}[\Phi(\e\eta,u)]$ implies $\cE_0[u] \leq \cE_0[\Phi(\e\eta,u)]$, and thus we can proceed as in \cite[Corollary 3.6]{FGKS1} to deduce \eqref{eq:proj-pde} (for each component $\Pi^i\circ u$, $i = 1,2,\cdots,m$). 
\end{proof}


\section{Partial regularity}\label{sec:part-reg}

This section is devoted to the proof of Theorem \ref{thm:e-reg}. Once this theorem is established, the (optimal) partial regularity theory, Corollary \ref{cor:partial}, follows from Theorem \ref{thm:e-reg-Ca}.

Our proof of Theorem \ref{thm:e-reg} follows, partly,  the classical approach of Giaquinta and Giusti \cite{GG}. The primary distinction in our setting is the choice of comparison maps: rather than harmonic functions, we employ $\cE_0$-minimizing constraint maps. This substitution ensures that the competitors remain compatible with the prescribed constraints. To this end, we establish the following maximum principle for $\cE_0$-minimizing constraint maps.

\begin{lem}
    \label{lem:max}
    There exists $\e_0 > 0$ small, depending only on $n$, $m$, and $\partial M$, such that for any minimizing constraint map $v\in W^{1,2}(B_1;\overline M)$ for $\cE_0$, if $ \int_{B_1}|Dv|^2\,dx \leq \e \leq \e_0$ and $\osc_{\partial B_1}(v-\ell) \leq \mu\sqrt\e$ for some $\mu > 1$ and some affine map $\ell :\R^n\to\R^m$ satisfying $|D\ell|\leq \mu\sqrt\e$, then 
    \begin{equation}
        \label{eq:max}
        \osc_{B_1} (v - \ell)\leq 2 \mu \sqrt \e.
    \end{equation}
\end{lem}

\begin{proof}
Suppose by way of contradiction that the conclusion is not true. Then there exists a sequence $\{v_k\}_{k=1}^\infty\subset W^{1,2}(B_1;\overline M)$ of minimizing constraint maps for $\cE_0$ such that 
\begin{equation}
    \label{eq:Dvk-L2}
    \int_{B_1} |Dv_k|^2\,dx \leq \e_k,\quad\text{and}\quad \osc_{\partial B_1} (v_k - \ell_k) \leq \mu_k\sqrt{\e_k},
\end{equation}
for some constants $\e_k\to 0$,  $\mu_k > 1$ and some affine map $\ell_k:\R^n\to\R^m$ with $|D\ell_k|\leq \mu_k\sqrt{\e_k}$, yet 
\begin{equation}
    \label{eq:vk-B1}
    \osc_{B_1} (v_k - \ell_k) \geq 2\mu_k\sqrt{\e_k},
\end{equation}
for every $k\in\N$. Rescale the target and set 
$$
w_k := \frac{v_k - \ell_k - y_k}{\mu_k\sqrt{\e_k}},
$$

where $y_k$ is chosen as the integral average of $v_k - \ell_k$ over $B_1$. By the Poincar\'e inequality and the fractional trace inequality \cite[Chapter 10, Theorem 18.1, Remark 18.2]{DBb}, we have   
$$\| w_k \|_{W^{\frac{1}{2},2}(\partial B_1)}^2 \leq c \| w_k \|_{W^{1,2}(B_1)}^2 \leq \frac{c}{\mu_k^2\e_k} \int_{B_1} |Dv_k|^2 + |D\ell_k|^2 \,dx < c,$$
where the last inequality follows from $\mu_k > 1$ and \eqref{eq:Dvk-L2}, and $c > 1$ here is a generic constant that depends at most on $n$. Hence, it follows $w_k\rightharpoonup w$ in $W^{1,2}(B_1;\R^m)\cap W^{\frac{1}{2},2}(\partial B_1;\R^m)$ for some $w\in W^{1,2}(B_1;\R^m)$ along a subsequence that we do not attempt to relabel. Extracting a further subsequence if necessary, we also obtain that $w_k \to w$ a.e.\ in $B_1$ and ($\cH^{n-1}$-)a.e.\ on $\partial B_1$. Thus, it follows from \eqref{eq:vk-B1} that 
\begin{equation}
    \label{eq:w-osc}
\osc_{B_1} w \geq 2 > 1 \geq \osc_{\partial B_1} w.
\end{equation}
However, since $\Delta v_k = A_{v_k}(Dv_k,Dv_k)\chi_{v_k^{-1}(\partial M)}$ in $B_1$ in the weak sense, denoting by $\kappa$ the uniform curvature bound for $\partial M$, we observe that for all $\vp\in W_0^{1,2}\cap L^\infty(B_1;\R^m)$, 
$$
\bigg|\int_{B_1} Dw_k : D\vp \,dx \bigg| \leq \frac{\kappa}{\mu_k\sqrt{\e_k}} \int_{B_1} |\vp||Dv_k|^2\,dx \leq \frac{\kappa\sqrt{\e_k}}{\mu_k} \sup_{B_1} |\vp|  \to 0,
$$
where the second inequality was deduced from \eqref{eq:Dvk-L2} and $\mu_k > 1$. Passing to the limit in the above inequalities, we observe that 
$$
\int_{B_1} Dw: D\vp\,dx = 0,\quad\forall \vp\in W_0^{1,2}\cap L^\infty(B_1;\R^m),
$$
which implies that $\Delta w = 0$ in $B_1$. Therefore, applying the maximum principle to each component of $w$, we arrive at a contradiction in \eqref{eq:w-osc}.
\end{proof}

We shall first prove an interior $C^{1,\alpha}$-estimate which is optimal for $\gamma \in (0,1)$, and suboptimal for $\gamma \in[1,2)$. 

\begin{lem}
    \label{lem:C1a-sub}
    Let $u\in W^{1,2}(B_4;\overline M)$ be a minimizing constraint map for $\cE_{\lambda,\gamma}$ with $\|\rho\circ u\|_{L^\infty(B_4)} \leq d$. Given $\sigma\in(0,1)$, there exist $\e > 0$ and $c>1$, both depending only on $n$, $m$, $\partial M$, $\lambda$,  $\gamma$, and $\sigma$, such that if $4^{2-n}\int_{B_4}|Du|^2\,dx \leq \e$ and ${\lambda} d^\gamma \leq \e$, then $u\in C^{1,\alpha}(B_1)$ with $\alpha = \min\{\frac{\gamma}{2-\gamma},\sigma\}$ and  
    $$
    \| Du \|_{C^{0,\alpha}(B_1)} \leq c.
    $$
\end{lem}

\begin{proof}
Fix ${\e}$ as a small constant, to be determined later, at most by the parameters $n$, $m$, $\partial M$, $\lambda$, $\gamma$ and $\sigma$. Also fix any $\sigma\in(0,1)$, and set $\alpha = \min\{\frac{\gamma}{2-\gamma},\sigma\}$ as in the statement. We shall write by $c > 1$ a generic constant that may vary at each occurrence.

Let $u\in W^{1,2}(B_4;\overline M)$ be a minimizing constraint map for $\cE_{\lambda,\gamma}$, and suppose the assumptions in the statement of the lemma is fulfilled, i.e., 
\begin{equation}
    \label{eq:Du-L2-e0}
    4^{2-n}\int_{B_4} |Du|^2\,dx \leq {\e}, \quad \|\rho\circ u\|_{L^\infty(B_4)}\leq d,\quad\text{and}\quad {\lambda} d^\gamma\leq {\e}.  
\end{equation} 

Fix ${R} \in (0,1)$ as well as $x_0\in   B_2$, and let $v \in W_u^{1,2}(B_{R}(x_0);\overline M)$ be a minimizing constraint map for $\cE_0$. Then by  the energy-minimality of $v$, the inequalities \eqref{eq:Du-L2-e0} and almost monotonicity (Lemma \ref{lem:monot}) we have 
\begin{equation}
    \label{eq:Dv-L2-e0}
    R^{2-n}\int_{B_R(x_0)} |Dv|^2\,dx \leq R^{2-n}\int_{B_R(x_0)} |Du|^2\,dx \leq 3^{2-n} \int_{ B_3} |Du|^2\,dx + c{\lambda} d^\gamma \leq c{\e}.
\end{equation}
Thus, with ${\e}$ small, depending only on $n$, $m$, and $\partial M$, it follows from Theorem \ref{thm:e-reg-clean}, especially \eqref{eq:camp} with $\bar\alpha := \min\{\frac{\gamma}{2-\gamma},\frac{1}{2}(1 +\sigma)\}$ which verifies $\alpha < \bar\alpha < 1$, that  
\begin{equation}
    \label{eq:v-C11-re}
    \begin{aligned}
        \int_{B_r(x_0)} |Dv - (Dv)_{x_0,r}|^2\,dx &\leq c\bigg(\frac{r}{R/2}\bigg)^{n+2\bar\alpha}\int_{B_{R/2}(x_0)} |Dv - (Dv)_{x_0,{R}}|^2\,dx \\
        &\quad+ cr^{n+2\bar\alpha}\int_{B_R(x_0)} |Dv|^2\,dx, \\
        &\leq c\bigg(\frac{r}{R}\bigg)^{n+2\bar\alpha}\int_{B_{R}(x_0)} |Dv - (Dv)_{x_0,{R}}|^2\,dx \\
        &\quad+ cr^{n+2\bar\alpha}\int_{B_R(x_0)} |Dv|^2\,dx,
    \end{aligned}
 \end{equation} 

whenever $0<r\leq R$. Now utilizing the $\cE_{\lambda,\gamma}$-minimality of $u$, as well as the fact that $v$ is a weak solution to $\Delta v = A_v(Dv,Dv)\chi_{v^{-1}(\partial M)}$ in $B_{R}(x_0)$, we see that
\begin{equation}
    \label{eq:Eu-Ev}
    \begin{aligned}
    \int_{B_{R}(x_0)} |D(u-v)|^2\,dx & = \int_{B_{R}(x_0)} |Du|^2 - |Dv|^2 + 2(u- v)\cdot A_v(Dv,Dv)\chi_{v^{-1}(\partial M)}\,dx\\
    &\leq {\lambda} \int_{B_{R}(x_0)} (\rho\circ v)^\gamma - (\rho\circ u)^\gamma\,dx + c \int_{B_{R}(x_0)} |u-v| |Dv|^2\,dx.
    \end{aligned}
\end{equation}
Since we assume that $\rho\circ u \leq d$ a.e.\ in $B_4$ (hence in $B_R(x_0)$), it follows from the subharmonicity of $\rho\circ v$ (see \cite[Lemma 3.2]{FGKS2}) and $u = v$ on $\partial B_R(x_0)$ that $\rho\circ v \leq d$ a.e.\ in $B_R(x_0)$. Since $\rho$ is the signed distance function to a smooth hypersurface $\partial M$, we have 
$$
\esssup_{B_d(\partial M)}|\nabla\rho| \leq 1, 
$$
so we deduce from the H\"older, Sobolev and Young inequality, as in \cite[Page 246]{GG2}, that
\begin{equation}
    \label{eq:I}
    \begin{aligned}
    \int_{B_{R}(x_0)} (\rho\circ v)^\gamma - (\rho\circ u)^\gamma\,dx &\leq c \int_{B_{R}(x_0)} |u-v|^{\min\{\gamma,1\}}\,dx \\
    &\leq c \bigg[\int_{B_{R}(x_0)} |u-v|^{\frac{2n}{n-2}}\,dx \bigg]^{\frac{\gamma(n-2)}{2n}} {R}^{n-\frac{\gamma(n-2)}{2}} \\
    &\leq c \bigg[\int_{B_{R}(x_0)} |D(u-v)|^2 \bigg]^{\frac{\gamma}{2}} {R}^{n-\frac{\gamma(n-2)}{2}} \\
    &\leq \frac{1}{2}\int_{B_{R}(x_0)} |D(u-v)|^2\,dx + c {R}^{n + 2\min\{\frac{\gamma}{2-\gamma},1\}}.
    \end{aligned}
\end{equation}
To estimate the last term in the second line of \eqref{eq:Eu-Ev}, let $\alpha_0\in(0,1)$ and $\delta > 0$ be some constants to be determined later by $\sigma$ alone. Choose ${\e}$ to be smaller, if necessary, for now depending only  on $n$, $m$, $\partial M$, $\lambda$, $\gamma$, $\alpha_0$ and $\delta$ (the dependence on $d$ is absorbed by that of $\lambda$ and $\gamma$ due to \eqref{eq:Du-L2-e0}), such that by \cite[Proposition 1]{L}, the small energy assumption in \eqref{eq:Du-L2-e0} implies $|Du|\in L^{2,n-2\delta}(B_1(x_0))$, and $u\in C^{0,\alpha_0}(B_1(x_0))$, with the estimate 
\begin{equation}
    \label{eq:u-Ca}
\| Du\|_{L^{2,n-2\delta}(B_1(x_0))}^2 + [u]_{C^{0,\alpha_0}(B_1)}^2 \leq c\int_{B_4} |Du|^2\,dx.
\end{equation}

Recalling the energy minimality of $v$ in $W_u^{1,2}(B_R(x_0);\overline M)$, we obtain from \eqref{eq:u-Ca} that 
\begin{equation}
    \label{eq:Dv-L2-re}
\int_{B_R(x_0)} |Dv|^2\,dx \leq \int_{B_R(x_0)} |Du|^2\,dx \leq cR^{n-2\delta}\int_{B_4}|Du|^2\,dx.
\end{equation}
Moreover, by \eqref{eq:Dv-L2-e0} (as well as the fact that $v = u$ on $\partial B_R(x_0)$), we can also employ Lemma \ref{lem:max} (here we may also need to take ${\e}$ smaller if needed) to deduce that  

\begin{equation}
    \label{eq:v-osc}
    \osc_{B_R(x_0)} v \leq 2 \bigg[\osc_{\partial B_R(x_0)} u\bigg] \leq 2 \bigg[\osc_{B_R(x_0)} u \bigg] \leq 2 
    c R^{\alpha_0}.
\end{equation}
Utilizing \eqref{eq:Dv-L2-re} and \eqref{eq:v-osc}, we estimate the rightmost term in \eqref{eq:Eu-Ev} as 
\begin{equation}
    \label{eq:II}
\int_{B_R(x_0)} |u-v||Dv|^2\,dx \leq \esssup_{B_R(x_0)}|u-v| \int_{B_R(x_0)} |Dv|^2\,dx \leq c R^{n-2\delta + \alpha_0}\int_{B_4} |Du|^2\,dx.
\end{equation}

Now let $\sigma \in(0,1)$ be  given. At this point, we can select $(\alpha_0,\delta)$ accordingly such that $\frac{1}{2}(1+\sigma) = \alpha_0 - 2\delta$ (e.g., by choosing $\delta = \frac{1}{8}(1-\sigma)$ and $\alpha_0 = \frac{1}{4}(3 + \sigma)$). Then by putting \eqref{eq:I} and \eqref{eq:II} together, and by choosing $\alpha_1:= \min\{\frac{\gamma}{2-\gamma},\frac{1}{4}(1+\sigma)\} < \min\{\frac{\gamma}{2-\gamma},1\}$, we obtain from \eqref{eq:Eu-Ev} that 
\begin{equation}
    \label{eq:Eu-Ev-re}
    \int_{B_{R}(x_0)} |D(u-v)|^2 \,dx \leq c{R}^{n + 2\alpha_1} \bigg(\int_{B_4}|Du|^2\,dx + 1\bigg) \leq cR^{n+2\alpha_1}
\end{equation}
Finally, we deduce from \eqref{eq:v-C11-re}, \eqref{eq:Dv-L2-re}, and \eqref{eq:Eu-Ev-re} that 
$$
    \begin{aligned}
        &\int_{B_r(x_0)} |Du - (Du)_{x_0,r}|^2\,dx \\
        &\leq 2\int_{B_r(x_0)} |Dv - (Dv)_{x_0,r}|^2 + 2 |D(u-v)|^2\,dx \\
        & \leq c \bigg(\frac{r}{R}\bigg)^{n+2\bar\alpha} \int_{B_{R}(x_0)}|Dv - (Dv)_{x_0,{R}}|^2\,dx+ cr^{n+2\bar\alpha}\int_{B_R(x_0)}|Dv|^2\,dx + c\int_{B_{R}(x_0)} |D(u - v)|^2\,dx\\
        & \leq c   \bigg(\frac{r}{R}\bigg)^{n+2\bar\alpha} \int_{B_{R}(x_0)} |Du - (Du)_{x_0,{R}}|^2\,dx + c{R}^{n + 2\alpha_1}.
    \end{aligned}
$$
whenever $0<r<R<1$.
Recall that $\bar\alpha = \min\{\frac{\gamma}{2-\gamma},\frac{1}{2}(1+\sigma)\}$, which satisfies $\alpha_1< \bar\alpha < 1$. Thus, it follows from the iteration lemma, c.f. \cite[Lemma 2.2]{GG}, that 
\begin{equation}
    \label{eq:Du-C1a}
    \begin{aligned}
    \int_{B_r(x_0)} |Du - (Du)_{x_0,r}|^2\,dx &\leq c \bigg(\frac{r}{R}\bigg)^{n+2\alpha_1} \int_{B_{R}(x_0)} |Du - (Du)_{x_0,{R}}|^2\,dx  + c  r^{n+2\alpha_1}, 
    \end{aligned}
\end{equation}
whenever $0<r<R<1$ and $x_0\in   B_2$. In particular, taking $R \uparrow 1$, and utilizing Jensen's inequality that yields $|(Du)_{x_0,1}|^2 \leq \fint_{B_1(x_0)} |Du|^2\,dx \leq c\int_{B_4} |Du|^2\,dx$, we observe from \eqref{eq:Du-C1a} that 
$$
\int_{B_r(x_0)} |Du - (Du)_{x_0,r}|^2\,dx \leq c r^{n+2\alpha_1},
$$ 

for all $r\in(0,1)$ and all $x_0\in   B_2$. Therefore, $u\in C^{1,\alpha_1}(  B_2)$ with the estimate
\begin{equation}
    \label{eq:Du-C1a-re}
\| Du\|_{C^{0,\alpha_1}(  B_2)}^2\leq c.
\end{equation}

Now if $\alpha_1 \geq \alpha = \min\{\frac{\gamma}{2-\gamma},\sigma\}$, then the proof is finished. Otherwise, that is, if $\alpha_1 < \alpha$, we perform a bootstrapping argument as follows. Due to \eqref{eq:Du-C1a-re} and \eqref{eq:Du-L2-e0}, we obtain  
\begin{equation}
    \label{eq:u-l}
\sup_{\partial B_R(x_0)}| u - \ell| \leq c R^{1+\alpha_1},
\end{equation}

for every $R \in (0,1)$, where $\ell$ is the affine part of $u$, i.e., $\ell(x_0) = u(x_0)$ and $D\ell = Du(x_0)$. Thus, we can update the oscillation estimate \eqref{eq:v-osc} of $v$ over $B_R(x_0)$, via Lemma \ref{lem:max} again, to  
\begin{equation}
    \label{eq:v-l}
    \sup_{B_R(x_0)} |v- \ell| \leq 2\bigg[\sup_{\partial B_R(x_0)} |u - \ell| \bigg] \leq 2cR^{1+\alpha_1}.
\end{equation} 
By \eqref{eq:u-l} and \eqref{eq:v-l}, we can improve \eqref{eq:II} as
\begin{equation}
    \label{eq:II-re}
    \int_{B_R(x_0)} |u-v||Dv|^2\,dx \leq  c R^{n  - 2\delta + 1 + \alpha_1}\int_{B_4}|Du|^2\,dx. 
\end{equation}
Thus, by choosing $\delta = \frac{1}{2}(1-\sigma)$, we can further improve \eqref{eq:Eu-Ev-re} by replacing $\alpha_1$ with $\alpha_2 = \min\{\frac{\gamma}{2-\gamma},\frac{3}{8}(1+\sigma)\}$, which then results a corresponding improvement in \eqref{eq:Du-C1a}.

Iterating the above argument $k$-times (which requires the smallness of ${\e}$ to be also dependent on $k$), we obtain \eqref{eq:Du-C1a} with $\alpha_k = \min\{\gamma/(2-\gamma),\frac{1}{2}\theta_k(1+\sigma)\}$ in place of $\alpha_1$, for some strictly increasing sequence $\theta_k\to 1$. In particular, we reach the optimal exponent $\alpha_k \geq \alpha = \min\{\gamma/(2-\gamma),\sigma\}$ within a finite time of iterations. This completes the proof. 
\end{proof}

We are left with the optimal regularity estimate for the case $\gamma\in[1,2)$. Recall from Section \ref{sec:prelim} that $d_0 > 0$ is a small constant, determined solely by the smoothness of $\partial M$, such that $\rho\in C^\infty(B_{d_0}(\partial M))$. 

\begin{lem}
    \label{lem:C1a-sup}
    Assume $\gamma \in [1,2)$, and let $u\in W^{1,2}(\Omega;\overline M)$ be a minimizing constraint map for $\cE_{\lambda,\gamma}$. If $u\in C^1(U; B_{d_0}(\partial M))$ for a subdomain $U\subset\Omega$, then $|D^2u| \in L_\loc^\infty(U)$. Moreover, with $\lambda_0 \geq 1$ chosen such that $(\diam U)\| Du\|_{L^\infty(U)} \leq \lambda_0$, one has, for every ball $B\Subset U$,
    $$
    \| D^2 u \|_{L^\infty(B)} \leq  c\norm{Du}_{L^\infty(U)}^2 + c\lambda \gamma (d_0)^{\gamma-1} 
    $$

    where $c>1$ depends only on $n$, $m$, $\partial M$, $\lambda$, $\gamma$,  $\lambda_0$, $d_0$, $\dist(B, \partial U)$.  
\end{lem}

\begin{proof}
    Since $u\in C^1(U;B_{d_0}(\partial M))$, Lemma \ref{lem:dist-pde} yields $\rho\circ u\in C^1(U)$. Moreover, since $\gamma \geq 1$, $(\rho\circ u)^{\gamma-1}\in C(U)$, whence it follows from \eqref{eq:dist-pde}, the assumption $|Du|\in L^\infty(U)$, and $\rho\in C^\infty(B_{d_0}(\partial M))$ (which along with $u \in C^1(U; B_{d_0}(\partial M))$ yields 
    $\Hess\rho_u(Du,Du) \in C(U)$), that $|\Delta (\rho\circ u)| \in L^\infty(U)$. Thus, the elliptic regularity estimate implies $\rho\circ u \in W_\loc^{2,p}(U)$ for every $p\in(1,\infty)$. 

    Similarly, since $\Pi\in C^\infty(B_{d_0}(\partial M);\partial M)$, and $\nu\in C^\infty(\partial M;\Ss^{m-1})$, we observe from $u\in C^1(U;B_{d_0}(\partial M))$ and Lemma \ref{lem:proj-pde} that $|\Delta(\Pi\circ u)| \in L^\infty(U)$. Therefore, we also have $\Pi\circ u \in W_\loc^{2,p}(U;\partial M)$ for every $p\in (1,\infty)$.

    Combining the above observations together, we deduce (from $\id = \Pi + \rho \nu\circ \Pi$ in $B_{d_0}(\partial M)$, c.f. \cite[Eq.\ (2.1)]{FGKS1}) that $u \in W_\loc^{2,p}(U)$ for every $p \in (1,\infty)$. To this end, we may argue as in the proof of \cite[Lemma A.3--A.4]{FKS} (with a very minor modification) to deduce that $u\in C_\loc^{1,1}(U)$.

    More precisely, from elliptic theory for $\rho \circ u$, $\Pi \circ u$ we have that $u\in C^{1,\alpha}(B)$. Hence, we have from \cite[(4.2)]{FKS}
\begin{equation*}
    \Delta (\Pi \circ u) = \Hess \Pi_u ((Du), (Du)) = f_v \in W^{1,p}(B) \quad \text{in} \; B.
\end{equation*}
Hence, $\Pi\circ u \in W^{3,p}_{loc}(B)$ for any $p<\infty$, and in particular $C^{1,1}_{loc}(B)$ (by Sobolev embedding for $p \geq n$)
\begin{equation*}
    \norm{\Pi \circ u}_{C^{1,1}(B)} \leq c \norm{\Pi\circ u}_{W^{3,p}(B)} \leq c \norm{f_v}_{W^{1,p}(B)}.
\end{equation*}
And for the distance function we have that ($\gamma \in [1,2)$)
\begin{equation*}
    \begin{split}
        \Delta (\rho \circ u) = \Hess\rho_u(Du,Du) + \lambda\gamma(\rho\circ u)^{\gamma-1} \\= (\Hess\rho_u(Du,Du) + \lambda\gamma(\rho\circ u)^{\gamma-1})\chi_{\{\rho \neq 0\}} = f_1 \in C^{0,\sigma}_{loc} (B).
    \end{split}
\end{equation*}
According to the result for the no-sign free boundary problem \cite{ALS} we obtain that $\rho \circ u \in C^{1,1}(B_{1/2})$. More precisely, if $f_1 = \Delta v$ (since $f_1 = 0$ in $\{\rho = 0\}$)
\begin{equation*}
    \norm{D^2 (\rho \circ u)}_{L^{\infty}(B_{1/2})} \leq c (\norm{\rho \circ u}_{L^1(B)} + \norm{D^2 v}_{L^\infty(B)}).
\end{equation*}

    
    The interior $C^{1,1}$-estimate  can be deduced by keeping track of the regularity estimate at each step above. 
\end{proof}

We are now ready to prove the $\e$-regularity theorem.

\begin{proof}
    [Proof of Theorem \ref{thm:e-reg}]
    The proof follows from Lemmas \ref{lem:C1a-sub} and \ref{lem:C1a-sup}, along with a standard covering argument. 
    First, note that the condition $\lambda d^\gamma \leq \varepsilon$ from Lemma 3.2 is utilized in \eqref{eq:Dv-L2-e0} where the almost monotonicity formula is applied. Consequently, when carrying out the proof on a ball of radius $R$, we obtain the assumption $R^2 \lambda d^\gamma \leq \varepsilon$.

    To obtain the first estimate in Theorem \ref{thm:e-reg} from \eqref{eq:Du-C1a-re}, we can use the rescaling $\Tilde{u} (x) = u (\frac{x}{R})$ where $x \in B_R$. Then $D \Tilde{u} (x) = \frac{1}{R} Du (y) \Big |_{y = \frac{x}{R}}$, and the estimate follows.

To obtain the second estimate we combine the estimate in Lemma \ref{lem:C1a-sup} with the one in Lemma \ref{lem:C1a-sub}, and the factor $R^2$ comes from scaling.

\end{proof}


\section*{Acknowledgements}
This project was     supported by  Henrik Shahgholian’s grant from  the Swedish Research Council (grant no.~2021-03700). I would like to express my sincere gratitude to Sunghan Kim for proposing the research problem and providing critical guidance throughout the development of the ideas and results presented in this paper.






\end{document}